\documentclass[12pt]{amsart}
\usepackage{enumitem}

\calclayout

\newcommand{\PP}{\mathbb{P}}
\newcommand{\OO}{\mathcal{O}}
\newcommand{\KK}{\mathbb{K}}
\newcommand{\mm}{\mathfrak{m}}

\newcommand{\Soc}{\textnormal{Soc}}

\newtheorem{theorem}{Theorem}[section]
\newtheorem{lemma}[theorem]{Lemma}

\theoremstyle{definition}
\newtheorem{definition}[theorem]{Definition}
\newtheorem{example}[theorem]{Example}

\theoremstyle{remark}

\theoremstyle{proposition}
\newtheorem{proposition}[theorem]{Proposition}

\theoremstyle{conjecture*}
\newtheorem*{conjecture*}{Conjecture}

\numberwithin{equation}{section}

\begin{document}

\title{Gorenstein algebras and uniqueness of additive actions}

%    Remove any unused author tags.

%    author one information
\author{Ivan Beldiev}
\address{HSE University, Faculty of Computer Science, Pokrovsky Boulevard 11, Moscow, 109028 Russia}
%\curraddr{}
\email{ivbeldiev@gmail.com}
\thanks{Supported by the Russian Science Foundation grant 23-21-00472}



\date{}

\dedicatory{}

\begin{abstract}
We study induced additive actions on projective hypersurfaces, i.e. regular actions of the algebraic group $\mathbb G_a^m$ with an open orbit that can be extended to a regular action on the ambient projective space. We prove that if a projective hypersurface admits an induced additive action, then it is unique if and only if the hypersurface is non-degenerate. We also show that for any $n\geq 2$, there exists a non-degenerate hypersurface in $\PP^n$ of each degree~$d$ from $2$ to $n$.
\end{abstract}

\maketitle

\section*{Introduction}
Let $\KK$ be an algebraically closed field of characteristic zero. From now on, by an algebraic variety we mean an algebraic variety over $\KK$. Denote by $\mathbb G_a^m$ the algebraic group $(\KK, +)^m$, where $m$ is a positive integer.

An \emph{additive action} on an algebraic variety $X$ is an effective regular action of the group $\mathbb G_a^m$ on $X$ with an open orbit. If $X\subseteq \PP^n$ is a projective hypersurface, then we can consider so-called \emph{induced additive actions} on $X$, i.e. additive actions of $\mathbb G_a^m$ on $X$ that can be extended to a regular action of $\mathbb G_a^m$ on the ambient projective space $\PP^n$. All additive actions we consider are induced, so we will sometimes omit the word "induced" speaking about them.

Not all additive actions on projective hypersurfaces are induced (see \cite[Example 2.2]{AZ}). However, the property of being induced is not as restrictive as it might seem.
For example, suppose that $X\subseteq \PP^{n}$ is a linearly normal subvariety, i.e. $X$ is not contained in any hyperplane and it is not a linear projection of a subvariety from a bigger
projective space. Equivalently, this means that the map $H^0(\PP^n, \OO(1)) \to H^0(X, \OO(1))$ is surjective. In this case, if $X$ admits an additive action, then this action is induced (for a proof, see \cite[Section~2]{AP}).

The development of the theory of additive actions began with the work \cite{HT} by Hassett and Tschinkel. They showed that additive actions on $\PP^n$ are in natural bijection with local finite-dimensional commutative associative unital algebras of dimension~$n + 1$. According to a generalized version of the Hassett-Tschinkel correspondence (see \cite[Section 1.5]{AZ}), there exists, up to natural equivalences, a bijection between the following objects:

\begin{enumerate}
    \item induced additive actions on projective hypersurfaces in $\mathbb P^n$ not contained in any hyperplane;
    \item pairs $(A, U)$, where $A$ is a local commutative associative unital algebra over $\KK$ of dimension $n$ with the maximal ideal $\mathfrak m$ and $U\subseteq \mathfrak m$ is a subspace of dimension $n-2$ generating the algebra $A$. Such pairs $(A,U)$ are called \emph{$H$-pairs}.
\end{enumerate}

In \cite{Sh}, it is shown that there is a unique, up to isomorphism, induced additive action of~$\mathbb G_a^n$ on the smooth quadric $Q_n\subseteq\PP^{n+1}$. In \cite{BGT}, the authors obtain a generalization of this result for regular actions of arbitrary commutative linear algebraic groups on $Q_n \subseteq \PP^{n+1}$. They show that, besides the unique additive action of $\mathbb G_a^n$, there are only the following three cases: $\mathbb G_m$-action on $Q_1$, $\mathbb G_a\times \mathbb G_m$-action on $Q_2$, and $\mathbb G_m^2$-action on $Q_2$ (here $\mathbb G_m$ is the multiplicative group~$(\KK, \times))$.

For singular quadrics, the situation becomes more complicated. In \cite[Chapter 4]{ASh} it is shown that there exists an infinite family of pairwise non-equivalent induced additive actions on quadrics of corank $1$ in $\PP^n$ for $n\geq 5$. In \cite{AP}, the authors classify the induced additive actions of $\mathbb G_a^n$ on quadrics of corank $1$ in~$\PP^{n+1}$. In \cite{YL}, a classification of additive actions on quadrics of corank $2$ such that their singularities are not fixed by the $\mathbb G_a^n$-action is obtained. In \cite{Shaf}, all projective toric hypersurfaces admitting an additive action are found. Namely, it is proved that a toric hypersurface admitting an additive action is isomorphic either to~$\PP^n$ or to a quadric of rank~$3$ or~$4$. Also, a complete classification of additive actions on quadrics of small dimensions is given in this work.

The case of non-degenerate hypersurfaces is of particular interest. A hypersurface $X\subseteq~\PP^n$ given by a homogeneous polynomial $f\in \KK[z_0, z_1, \ldots, z_n]$ is called \emph{non-degenerate} if there is no linear transform of variables reducing the number of variables in $f$. Equivalently, this means that $X$ is not a projective cone over a hypersurface $X_0\subseteq \PP^k$ in a projective subspace $\PP^k \subseteq \PP^n$ for some $k < n$. In particular, a quadric is non-degenerate if and only if it is smooth.

In \cite{AZ}, the authors establish a relation between additive actions on non-degenerate hypersurfaces and Gorenstein local finite-dimensional algebras, i.e. local finite-dimensional algebras $A$ with the maximal ideal $\mathfrak m$ such that the ideal $\Soc A = \{a\in A\mid a \mm = 0\}$ is one-dimensional. They show that induced additive actions on non-degenerate hypersurfaces in $\PP^n$ are in one-to-one correspondence with $H$-pairs $(A,U)$ such that $A$ is a Gorenstein local finite-dimensional algebra and $U\subseteq \mathfrak m$ is a complementary hyperplane to $\Soc A$. Moreover, it is shown in the same article that there exists at most one induced additive action on any non-degenerate hypersurface $X$ up to equivalence. There is the following natural conjecture (see \cite[Conjecture 5.19]{AZ}).

\begin{conjecture*}
    Let $X \subseteq \mathbb P^{n}$ be a degenerate hypersurface admitting an induced additive action. Then there are at least two non-equivalent induced additive actions on $X$.
\end{conjecture*}

In Theorem 2.2, we prove this conjecture. In order to do this, we apply the reduction procedure described in Proposition 1.9 below to the $H$-pair $(A, U)$ corresponding to a given degenerate hypersurface $X\subseteq \PP^n$ admitting an additive action. This gives us an $H$-pair~$(A_0, U_0)$ corresponding to a non-degenerate hypersurface $X_0\subseteq \PP^k$ such that $X$ is a projective cone over~$X_0$. After this, we construct two non-equivalent $H$-pairs $(A_1, U_1)$ and $(A_2, U_2)$ such that $\dim A_1 = \dim A_2 = \dim A_0 + 1$ that go to $(A_0, U_0)$ under reduction. 
\vspace{0.3cm}

In \cite[Corollary 2.16]{AZ}, the authors also show that the degree of a hypersurface $X\subseteq \PP^n$ admitting an induced additive action is at most $n$. It is noted that for $2\leq n\leq 5$, there are hypersurfaces in $\PP^n$ admitting an induced additive action of each degree from $2$ to $n$. In Proposition 3.1, we prove that this is true for any $n\geq 2$. The corresponding hypersurfaces are described in terms of $H$-pairs.

\section*{Acknowledgements}
The author is grateful to Ivan Arzhantsev and Yulia Zaitseva for useful discussions.

\section{Preliminaries}
Let us begin with some basic results on local finite-dimensional algebras. From now on, when we say "algebra", we mean a finite-dimensional commutative associative unital algebra over the field $\KK$. Recall that an algebra is called \emph{local} if it contains a unique maximal ideal~$\mathfrak m$.

\begin{lemma} \cite[Lemma 1.2]{AZ}
    A finite-dimensional algebra $A$ is local if and only if $A$ is the direct sum of its subspaces $\mathbb K \oplus \mathfrak m$, where $\mathfrak m$ is the ideal consisting of all nilpotent elements of $A$.
\end{lemma}

\begin{definition}
    The \emph{socle} of a local algebra $A$ with the maximal ideal $\mathfrak m$ is the ideal $$\Soc A = \{a\in A\mid a \mm = 0\}.$$ A local finte-dimensional algebra $A$ is called \emph{Gorenstein} if $\dim \Soc A = 1$.
\end{definition}

Consider the greatest positive integer $d$ such that $\mathfrak m^d \ne 0$. It is clear that $\mathfrak m^d\subseteq \Soc A$, but this inclusion can be strict. It follows that $A$ is Gorenstein if and only if $\dim \mathfrak m^d = 1$ and~$\Soc A = \mathfrak m^d$.

We already introduced the notion of an induced additive action on a hypersurface (see Introduction). Let us now give a formal definition of the equivalence of two actions.

\begin{definition}
    Two induced additive actions $\alpha_i\colon \mathbb G_a^m \times X_i\to X_i$, $X_i\subseteq \PP^n$, $i = 1, 2$, are called \emph{equivalent} if there exists an automorphism of algebraic groups $\varphi\colon \mathbb G_a^m \to \mathbb G_a^m$ and an automorphism $\psi \colon \PP^n \to \PP^n$ such that $\psi(X_1) = X_2$ and $$\psi~\circ~\alpha_1~=~\alpha_2~\circ~(\varphi\times~\psi).$$
\end{definition}

In the following definition, we consider pairs $(A, U)$, where $A$ is a local finite-dimensional algebra with the maximal ideal $\mm$ and $U\subseteq \mm$ is a subspace of codimension $1$ generating the algebra $A$.

\begin{definition}
    Two pairs $(A_1, U_1)$ and $(A_2, U_2)$ are called \emph{equivalent} if there exists an isomorphism of algebras $\varphi\colon A_1 \to A_2$ such that $\varphi(U_1) = U_2$.
\end{definition}

Now, we are ready to give a precise statement of the generalized version of the Hassett-Tschinkel correspondence, which we already mentioned in the introduction.

\begin{theorem}\cite[Theorem 2.6]{AZ}
    Suppose that $n\in \mathbb Z_{\geq 0}$. There is a one-to-one correspondence between the following objects:
    \begin{enumerate}
        \item induced additive actions on hypersurfaces in $\PP^{n-1}$ not contained in any hyperplane;
        \item pairs $(A,U)$, where $A$ is a local commutative associative unital algebra of dimension~$n$ and the maximal ideal $\mathfrak m$ and $U\subseteq \mathfrak m$ is a hyperplane generating the algebra $A$. Such pairs $(A,U)$ are called \emph{$H$-pairs}.
    \end{enumerate}
    This correspondence is considered up to equivalences from Definitions 1.3 and 1.4.
\end{theorem}

Let us sketch the construction of this correspondence. Suppose that we are given a pair $(A, U)$ satisfying the conditions of Theorem 1.5. Let $p\colon A \setminus \{0\} \to \PP(A)\cong \PP^{n-1}$ be the canonical projection. The corresponding projective hypersurface $X$ is defined as follows:

$$X = p(\overline{\KK^{\times}\exp U}),$$
i.e. it is the projectivization of the Zariski closure of the subset $\KK^{\times}\exp U\subseteq A\setminus\{0\}$. It is clear that $\exp(U)$ can be identified with $\mathbb G_a^{n - 2}$, hence the multiplication by the elements of $\exp(U)$ defines an action of $\mathbb G_a^{n - 2}$ on $\PP(A)$. It is easy to see that $X$ is preserved under this multiplication, so this gives us an induced additive action of $\mathbb G_a^{n - 2}$ on $X \subseteq \PP(A) \cong \PP^{n-1}$ as desired.

Conversely, suppose that we are given an induced additive action of $\mathbb G_a^{n-2}$ on a hypersurface $X\subseteq \PP^{n-1} = \PP(V)$, where $V$ is an $n$-dimensional vector space. It can be shown that this action can be lifted to an action of $\mathbb G_a^{n-2}$ on $V$, which gives us a faithful representation $\rho \colon \mathbb G_a^{n-2} \to \textnormal{GL}_m(\KK)$. Define $A$ as the subalgebra of $\textnormal{Mat}_m(\KK)$ generated by $\rho(\mathbb G_a^{n-2})$ and $U$ as the vector subspace of $A$ spanned by $\rho(\mathbb G_a^{n-2})$. One can check that $(A, U)$ is a pair satisfying the conditions of Theorem 1.5. The details can be found in \cite[Theorem~1.38]{AZ}.

%\begin{definition}
    %The \emph{H-pair} corresponding to an induced additive action on a hypersurface $X \subseteq \PP^{n+1}$ is the corresponding pair $(A, U)$, where $A$ is a local commutative associative unital algebra of dimension $n + 2$ with the maximal ideal $\mathfrak m$ and $U \subseteq \mathfrak m$ is a subspace of dimension~$n$ generating the algebra $A$.
%\end{definition}

Given an $H$-pair $(A, U)$, one can find the equation of the corresponding projective hypersurface as follows (for more details, see \cite[Chapter 2.2]{AZ}). Denote by $\pi$ the canonical projection $\pi\colon U \to U/\mm$. Let $d$ be the greatest positive integer such that $\mm^d \ne 0$. Then the corresponding projective hypersurface is given by the homogeneous polynomial 

$$z_0^d\pi(\ln(1 + \frac{z}{z_0})) = 0$$

for $z_0 + z\in A = \KK \oplus \mm$, $z_0\in\KK$, $z\in \mm$. This hypersurface is irreducible and has degree $d$.

In Sections 2 and 3, we illustrate this with several examples.

\begin{definition}
    Suppose that a projective hypersurface $X\subseteq \PP^n$ is given by the equation $f(z_0, z_1, \ldots, z_n)~=~0$, where $f$ is a homogeneous polynomial. The hypersurface $X$ is called \emph{non-degenerate} if one of the following equivalent conditions holds:
    \begin{enumerate}
    \item there exists no linear transform of variables reducing the number of variables in $f$;\\
    \item the hypersurface $X$ is not a projective cone over a hypersurface $X_0\subseteq \PP^k$ in a projective subspace $\PP^k \subseteq \PP^n$ for some $k < n$.
    \end{enumerate}
\end{definition}

The following theorem was proved in \cite[Theorem 2.30]{AZ}.

\begin{theorem}
    Induced additive actions on hypersurfaces of degree $d$ in $\PP^{n+1}$ are in one-to-one correspondence with pairs $(A, U)$, where $A$ is a Gorenstein local algebra of dimension~$n+2$ with the socle $\mathfrak m^d$ and $U\subseteq \mathfrak m$ is a complementary hyperplane to $\Soc A$.
\end{theorem}

Another theorem proved in \cite[Theorem 2.32]{AZ} is the following.

\begin{theorem}
    Suppose that $X\subseteq \mathbb P^{n-1}$ is a non-degenerate hypersurface. Then there exists at most one induced additive action on $X$ up to equivalence.
\end{theorem}

In \cite{AZ}, the authors introduce the notion of reduction of an induced additive action and prove its properties. Let an $H$-pair $(A, U)$ correspond to an induced additive action on a projective hypersurface $X \subseteq \PP^{n}$. Let $J$ be an ideal of $A$ of dimension $n - k$ such that $J\subseteq U$. 

\begin{proposition}\cite[Proposition 2.20 and Corollary 2.23]{AZ}
    The pair $(A/J, U/J)$ corresponds to an induced additive action on a projective hypersurface $X_0\subseteq \PP^{k}$, and $X$ is a projective cone over $X_0$, i.e. for some choice of coordinates in $\PP^{n}$ and $\PP^{k}$ the equations of the hypersurfaces~$X$ and $X_0$ are the same. Moreover, if $J$ is the maximal (with respect to inclusion) ideal of~$A$ contained in~$U$, then $X_0$ is a non-degenerate hypersurface in $\PP^{k}$.
\end{proposition}

\section{Non-uniqueness of additive action on degenerate hypersurfaces.} In this section, we prove the conjecture stated in the introduction. Let us first recall the statement. Suppose that $X\subseteq \PP^n$ is a degenerate hypersurface admitting an induced additive action. We are going to show that there exist at least two non-equivalent induced additive actions on $X$.

Let $(A, U)$ be the $H$-pair corresponding to the given induced additive action on $X$. Since $X$ is degenerate we can apply the reduction procedure from Proposition 1.9. This gives us a new $H$-pair $(A_0, U_0)$ such that $\dim A_0 < \dim A$.

The crucial step is the following. We need to construct two non-equivalent $H$-pairs $(A_1, U_1)$ and $(A_2, U_2)$ that go to $(A_0, U_0)$ under reduction and such that $$\dim A_1 = \dim A_2 = \dim A_0 + 1.$$ We are going to give an explicit construction of $(A_1, U_1)$ and $(A_2, U_2)$.

Denote by $\mm_0$ the maximal ideal of $A_0$. Choose $a_1, a_2, \ldots, a_k\in \mathfrak m_0\subseteq A_0$ such that the classes of these $k$ elements modulo $\mm_0^2$ form a basis of $\mm_0/\mm_0^2$. Then by Nakayama's Lemma, $a_1, a_2, \ldots, a_k$ generate the algebra~$A_0$. So, $A_0$ is isomorphic to $\KK[x_1, x_2, \ldots, x_k]/I$ for some ideal $I\subseteq \KK[x_1, x_2, \ldots, x_k]$; under this isomorphism, $x_1, x_2, \ldots, x_k$ are sent to $a_1, a_2, \ldots, a_k$, respectively. Since $a_1, a_2, \ldots, a_k$ form a basis of $\mm_0/\mm_0^2$, each $f\in I$ contains no linear term (the constant term of $f$ is also $0$).

We construct $(A_1, U_1)$ as follows. Let $A_1$ be the quotient of the polynomial algebra $\KK[x_1, x_2, \ldots, x_k, x_{k+1}]$ modulo the ideal $I_1$ defined as $$I_1 = I\KK[x_1, x_2, \ldots, x_k, x_{k+1}] + (x_1x_{k+1}) + (x_2x_{k+1}) + \ldots + (x_kx_{k+1}) + (x_{k+1}^2).$$
Informally speaking, we add a new variable $x_{k+1}$ to the algebra $A_0$ subject to the relations $x_ix_{k+1} = 0$ for any $i = 1, 2, \ldots, k, k+1$.

Clearly, $A_1 = A_0 \oplus \langle x_{k+1}\rangle$, so $\dim A_1 = \dim A_0 + 1$. We define $U_1$ as $U_1 = U_0 \oplus \langle x_{k+1}\rangle$. Now, $\langle x_{k+1}\rangle$ is an ideal of $A_1$ contained in $U_1$ and it is clear that if we apply the reduction of the $H$-pair $(A_1, U_1)$ with respect to this ideal, we obtain the $H$-pair $(A_0, U_0)$. So, the pair $(A_1, U_1)$ satisfies our conditions.

The construction of $(A_2, U_2)$ is more complicated. The informal idea is to shrink the ideal~$I$ instead of adding a new variable to $A_0$. In order to do this, we need the following technical lemma.

\begin{lemma}
    Let $J\subseteq \KK[x_1, x_2, \ldots, x_k]$ be a non-zero ideal. Then there exists a system $f_1, f_2, \ldots, f_l$ of generators of $J$ such that the ideal generated by $$f_1, f_2, \ldots, f_{l-1}, x_1f_l, x_2f_l, \ldots, x_kf_l$$ is strictly contained in~$J$.
\end{lemma}

\begin{proof}
    Choose an arbitrary finite system $f_1, f_2, \ldots, f_l$ of generators of the ideal $J$. If the ideal generated by $f_1, f_2, \ldots f_{l-1}, x_1f_l, x_2f_l, \ldots, x_kf_l$ is strictly contained in $J$, we are done. Otherwise, $$f_1, f_2, \ldots f_{l-1}, x_1f_l, x_2f_l, \ldots, x_kf_l$$ is another system of generators of $J$, so we can apply our procedure to this system replacing it with $$f_1, f_2, \ldots, f_{l-2}, x_1f_{l-1}, x_2f_{l-1}, \ldots, x_kf_{l-1}, x_1f_l, x_2f_l, \ldots, x_kf_l.$$
    We can proceed in a similar way, replacing at each step some $f_i$ with $x_1f_i, x_2f_i, \ldots, x_kf_i$. If the ideal $J$ becomes smaller at some step, we are done. Otherwise, after the $l$-th step we obtain a system
    $$x_1f_1, x_2f_1, \ldots, x_kf_1; x_1f_2, x_2f_2, \ldots, x_kf_2; \ldots x_1f_l, x_2f_l, \ldots, x_kf_l$$
    generating the ideal $J$ by our assumption.
    
    Denote by $d$ the smallest positive integer such that at least one of the polynomials $f_1, f_2, \ldots, f_k$ contains a non-zero term of degree $d$. Then there is no polynomial of the form $x_if_j$ ($i = 1, 2, \ldots, k$, $j = 1, 2, \ldots, l)$ containing a non-zero term of degree less than~$d+1$. If the polynomials of this form generate $J$, then any element of $J$ contains no term of degree less than $d+1$. This gives us a contradiction, so the proof is finished.
\end{proof}

We apply Lemma 2.1 to the ideal $I\subseteq \KK[x_1, x_2, \ldots, x_k]$. Choose a system $f_1, f_2, \ldots, f_l$ of generators of $I$ such that the ideal $\tilde I = (f_1, f_2, \ldots f_{l-1}, x_1f_l, x_2f_l, \ldots, x_kf_l)$ is strictly contained in $I$. Define an algebra $A_2$ as $A_2 = \KK[x_1, x_2, \ldots, x_k]/\tilde{I}$. It is not difficult to see that $\dim A_2 = \dim A_0 + 1$. Indeed, choose a system of polynomials $S_1, S_2, \ldots, S_r$ such that their classes modulo $I$ form a basis of $A_0$. Then any polynomial $g\in\KK[x_1, x_2, \ldots, x_k]$ can be written as $$g = a_1S_1 + a_2S_2 + \ldots + a_rS_r + H,$$ where $a_1, a_2, \ldots, a_r\in \KK$, $H\in I$. Write $H$ as $$H = h_1f_1 + h_2f_2 + \ldots + h_lf_l.$$ If the constant term of $h_l$ is $c$, then $$g = a_1S_1 + a_2S_2 + \ldots a_rS_r + cf_l + \tilde H,$$ where $\tilde H = h_1f_1 + h_2f_2 + \ldots + (h_l-c)f_l$ is an element of $\tilde I$. It follows that the classes of $S_1, S_2, \ldots, S_r, f_l$ generate $A_2$ as a vector space, so $\dim A_2 \leq \dim A_0 + 1$. But $\tilde I$ is strictly contained in $I$, so the equality $\dim A_2 = \dim A_0 + 1$ holds.

Let $U_2 = U_0 \oplus \langle f_l \rangle$. The element $f_l\in A_2$ is annihilated after multiplication by any $x_i$, $i~=~1,2, \ldots, k$, so $\langle f_l \rangle$ is an ideal of $A_2$ contained in $U_2$. It is clear that the reduction of $(A_2, U_2)$ modulo this ideal gives us the $H$-pair $(A_0, U_0)$.

Denote by $\mm_i$ the maximal ideal of $A_i$, $i = 1,2$. We see that 
$$\dim \mm_1/\mm_1^2 = \dim \mm_0/\mm_0^2 + 1$$ and $$\dim \mm_2/\mm_2^2 = \dim \mm_0/\mm_0^2.$$ This proves that the $H$-pairs $(A_1, U_1)$ and $(A_2, U_2)$ are not equivalent.

Our conjecture is now proved in the case when $\dim A - \dim A_0 = 1$. It remains to treat the case when $\dim A - \dim A_0 = r\geq 2$, i.e. we need to construct two non-equivalent $H$-pairs $(\widetilde{A}, \widetilde U)$, $(\hat{A}, \hat{U})$ sent to $(A_0, U_0)$ under the reduction and such that $$\dim \widetilde A = \dim\hat{A} = \dim A_0 + k.$$ To construct the $H$-pair $(\widetilde A, \widetilde U)$, we apply the first procedure of our construction $k$ times, i.e. we add $k$ new variables to $A_0$. To construct the $H$-pair $(\hat A, \hat U)$, we add $k-1$ new variables to $A_2$. If we denote by $\widetilde\mm$, $\hat\mm$ the maximal ideals of $\widetilde A$, $\hat A$ respectively, then $$\dim \widetilde \mm/\widetilde\mm^2 = \dim \hat\mm/\hat\mm^2 + 1,$$ so the $H$-pairs $(\widetilde A, \widetilde U)$ and $(\hat A, \hat U)$ are indeed non-equivalent.
\vspace{0.3cm}

Combining our result with Theorems 1.7 and 1.8, we can now state the following theorem.

\begin{theorem}
    Suppose that $X\subseteq \PP^n$ is a hypersurface that admits an induced additive action. Then the following conditions are equivalent:
    \begin{enumerate}
        \item there is a unique induced additive action on $X$;
        \item $X$ is a non-degenerate hypersurface;
        \item the $H$-pair $(A,U)$ corresponding to the given induced additive action on $X$ is such that the algebra~$A$ is Gorenstein with the maximal ideal $\mm$ and $U\subseteq \mm$ is a hyperplane complementary to $\Soc A$.
    \end{enumerate}
\end{theorem}

\begin{example}
    Consider the local algebra $A = \KK[x,y]/(x^4, x^2y, x^3 - y^2)$. Its maximal ideal~$\mm$ has a basis $x, y, x^2, xy, x^3$ (in particular, $\dim A = 6$). The subspace $U\subseteq \mm$ defined as $U = \langle x, y, x^2, xy\rangle$ generates the algebra $A$, so $(A, U)$ is an $H$-pair. It is easy to check that $\Soc A = \langle xy, x^3\rangle$. So, $A$ is not a Gorenstein algebra, and the corresponding projective hypersurface $X$ in $\PP^5$ is degenerate. Let us find the equation of this hypersurface. As before, denote by $\pi$ the canonical projection $\pi\colon \mm \to \mm/U$. In coordinate terms, $\pi(z)$ is simply the coefficient of $x^3$ in $z \in \mm$ for the unique expression of $z$ as a linear combination of the elements $x, y, x^2, xy, x^3$.
\begin{multline*} 
\pi(\ln(1 + t_1x + t_2y + t_3x^2 + t_4xy + t_5x^3)) = \\
= \pi(t_1x + t_2y + \ldots + t_5x^3 - \frac{(t_1x + t_2y + \ldots + t_5x^3)^2}{2} + \frac{(t_1x + t_2y + \ldots + t_5x^3)^3}{3}) = \\
= t_5 - t_1t_3 - \frac{t_2^2}{2} + \frac{t_1^3}{3},
\end{multline*}
    so the homogeneous equation of $X$ is
    $$z_0^2z_5 - z_0z_1z_3 - \frac{z_0z_2^2}{2} + \frac{z_1^3}{3} = 0.$$

Note that this equation does not depend on $z_4$, so the hypersurface $X$ is indeed degenerate.

Now, the subspace $\langle xy \rangle \subseteq U$ is an ideal of $A$ contained in $U$, so the reduction procedure can be applied. After the reduction, we obtain the algebra $$A_0 = \KK[x,y]/(x^4, xy, x^3 - y^2) = \KK[x,y]/(xy, x^3 - y^2)$$ with the maximal ideal $\mm_0 = \langle x, y, x^2, x^3\rangle$ and the generating subspace~$U_0 = \langle x, y, x^2 \rangle$.

Following the construction of the $H$-pair $(A_1, U_1)$ from the proof of Conjecture, we add a new variable $w$ to $A_0$. This gives us the algebra $$A_1 = \KK[x,y,w]/(xy, x^3 - y^2, xw, yw, w^2)$$ and its generating subspace $U_1 = \langle x, x^2, y, w\rangle$.

In order to find the equation of the corresponding hypersurface $X_1$, we do the following computations:
\begin{multline*} 
\pi_1(\ln(1 + t_1x + t_2y + t_3w + t_4x^2 + t_5x^3)) = \\
= \pi_1(t_1x + t_2y + \ldots + t_5x^3 - \frac{(t_1x + t_2y + \ldots + t_5x^3)^2}{2} + \frac{(t_1x + t_2y + \ldots + t_5x^3)^3}{3}) = \\
= t_5 - t_1t_4 - \frac{t_2^2}{2} + \frac{t_1^3}{3},
\end{multline*}

so the homogeneous equation of $X_1$ is
$$z_0^2z_5 - z_0z_1z_4 - \frac{z_0z_2^2}{2} + \frac{z_1^3}{3} = 0.$$

To construct $(A_2, U_2)$, we need to choose a system $f_1, f_2, \ldots, f_n$ of generators of the ideal $I = (xy, x^3 - y^2)$ and replace $f_n$ with $xf_n$, $yf_n$. The two polynomials $xy, x^3 - y^2$ generate the ideal $I$. If we apply our procedure to $xy$, we obtain the ideal $$\tilde{I} = (x^2y, xy^2, x^3 - y^2) = (x^2y, x^4, x^3 - y^2).$$ This gives us the algebra $$A_2 = \KK[x,y]/(x^2y, x^4, x^3 - y^2)$$ and its generating subspace $U_2 = \langle x, y, xy, x^2\rangle$. We see that it is the $H$-pair from which we started, so the corresponding hypersurface $X_2$ is the hypersurface $X$.
\bigskip

Let us write the formulas for the $\mathbb G_a^4$-actions on the hypersurfaces $X_1$ and $X_2$ explicitly. These actions are induced, and we are going to find the formulas for the corresponding $\mathbb G_a^4$-actions on~$\mathbb P^5$. In order to do this, we identify $\mathbb G_a^4$ with $\exp(U_i)$, $i = 1, 2$. Under this identification, the actions of $\mathbb G_a^4$ are given by the multiplication by elements of $\exp(U_i)$ for~$i = 1, 2$. 

For the $H$-pair $(A_1, U_1)$ an element $(t_1, t_2, t_3, t_4)\in \mathbb G_a^4$ is identified with $$\exp(t_1x + t_2y + t_3w + t_4x^2) \in \exp(U_1).$$ The computations we need to do are the following:
\begin{multline*} 
\exp(t_1x + t_2y + t_3w + t_4x^2)\cdot (z_0 + z_1x + z_2y + z_3w + z_4x^2 + z_5x^3) = \\
= (1 + t_1x + t_2y + t_3w + (t_4 + \frac{t_1^2}{2})x^2 + (\frac{t_2^2}{2} + \frac{t_1^3}{6})x^3)\cdot (z_0 + z_1x + z_2y + z_3w + z_4x^2 + z_5x^3) = \\
= z_0 + (z_1 + t_1z_0)x + (z_2 + t_2z_0)y + (z_3 + t_3z_0)w + (z_4 + t_1z_1 + (t_4 + \frac{t_1^2}{2})z_0)x^2 + \\
+ (z_5 + t_1z_4 + t_2z_2 + (t_4 + \frac{t_1^2}{2})z_1 + (\frac{t_2^2}{2} + \frac{t_1^3}{6})z_0)x^3,
\end{multline*}

so the action of $\mathbb G_a^4$ on $\PP^5$ is given by the formula
$$(t_1, t_2, t_3, t_4)\cdot [z_0: z_1: z_2: z_3: z_4: z_5] = $$
$$ = [z_0: z_1 + t_1z_0: z_2 + t_2z_0: z_3 + t_3z_0: z_4 + t_1z_1 + (t_4 + \frac{t_1^2}{2})z_0: z_5 + t_1z_4 + t_2z_2 + (t_4 + \frac{t_1^2}{2})z_1 + (\frac{t_2^2}{2} + \frac{t_1^3}{6})z_0].$$

The restriction of this action to the hypersurface $X_1\subseteq \PP^4$ has an open orbit isomorphic to the affine space $\mathbb A^4$. This orbit consists of all points $[z_0:z_1:z_2:z_3:z_4:z_5]\in X_1$ such that $z_0 \ne 0$. It is easy to see from the equation of $X_1$ that the complement to the open orbit in $X_1$ is the subspace
$$\{[z_0:z_1:z_2:z_3:z_4:z_5]\in \PP^5\colon z_0 = z_1 = 0\}.$$

On this set, the $\mathbb G_a^4$-action reduces to
$$(t_1, t_2, t_3, t_4)\cdot [0: 0: z_2: z_3: z_4: z_5] = [0: 0: z_2: z_3: z_4: z_5 + t_1z_4 + t_2z_2].$$

We obtain that the orbit of any point $[0: 0: z_2: z_3: z_4: z_5]$ with $z_2 \ne 0$ or $z_4\ne 0$ is isomorphic to $\mathbb A^1$, while the points $[0:0:0:z_3:0:z_4]$ are the fixed points of the $\mathbb G_a^4$-action. It follows that the set of fixed points of the $\mathbb G_a^4$-action is isomorphic to $\PP^1$.

\vspace{0.2cm}
Next, consider the $H$-pair $(A_2, U_2)$. For this $H$-pair, the element $(t_1, t_2, t_3, t_4)\in \mathbb G_a^4$ is identified with $$\exp(t_1x + t_2y + t_3x^2 + t_4xy) \in \exp(U_2).$$ The computations we need to do are the following:
\begin{multline*} 
\exp(t_1x + t_2y + t_3x^2 + t_4xy)\cdot (z_0 + z_1x + z_2y + z_3x^2 + z_4xy + z_5x^3) = \\
= (1 + t_1x + t_2y + (t_3 + \frac{t_1^2}{2})x^2 + (t_4 + t_1t_2)xy + (t_1t_3 + \frac{t_2^2}{2} + \frac{t_1^3}{6})x^3)\cdot (z_0 + z_1x + z_2y + z_3x^2 + z_4xy + z_5x^3) = \\
= z_0 + (z_1 + t_1z_0)x + (z_2 + t_2z_0)y + (z_3 + t_1z_1 + (t_3 + \frac{t_1^2}{2})z_0)x^2 + (z_4 + t_1z_2 + t_2z_1 + (t_4 + t_1t_2)z_0)xy + \\
+ (z_5 + t_1z_3 + t_2z_2 + (t_3 + \frac{t_1^2}{2})z_1 + (t_1t_3 + \frac{t_2^2}{2} + \frac{t_1^3}{6})z_0)x^3,
\end{multline*}

so the action of $\mathbb G_a^4$ on $\PP^5$ is given by the formula

$$(t_1, t_2, t_3, t_4)\cdot [z_0: z_1: z_2: z_3: z_4: z_5] = $$
$$ = [z_0: z_1 + t_1z_0: z_2 + t_2z_0: z_3 + t_1z_1 + (t_3 + \frac{t_1^2}{2})z_0: z_4 + t_1z_2 + t_2z_1 + (t_4 + t_1t_2)z_0:$$
$$ z_5 + t_1z_3 + t_2z_2 + (t_3 + \frac{t_1^2}{2})z_1 + (t_1t_3 + \frac{t_2^2}{2} + \frac{t_1^3}{6})z_0].$$

Again, the restriction of this action to the hypersurface $X_2\subseteq \PP^4$ has an open orbit isomorphic to the affine space $\mathbb A^4$ and consisting of all points $[z_0:z_1:z_2:z_3:z_4:z_5]\in X_2$ such that $z_0 \ne 0$. As for $(A_1, U_1)$, we find that the complement of this orbit in $X_2$ is
$$\{[z_0:z_1:z_2:z_3:z_4:z_5]\in \PP^5\colon z_0 = z_1 = 0\}.$$

On this set, the $\mathbb G_a^4$-action reduces to
$$(t_1, t_2, t_3, t_4)\cdot [0: 0: z_2: z_3: z_4: z_5] = [0: 0: z_2: z_3: z_4 + t_1z_2: z_5 + t_1z_3 + t_2z_2].$$

It is easy to see now that the the orbit of any point $[0:0:z_2:z_3:z_4:z_5]$ with $z_2 \ne 0$ is isomorphic to $\mathbb A^2$ and the orbit of any point $[0:0:0:z_3:z_4:z_5]$ with $z_3 \ne 0$ is isomorphic to $\mathbb A^1$. The fixed points are $[0:0:0:0:z_4:z_5]$ for any $z_4, z_5$ such $z_4 \ne 0$ or $z_5\ne 0$. The set of fixed poins of the $\mathbb G_a^4$-action is again isomorphic to $\PP^1$.

\bigskip
It is natural to ask if the construction of shrinking the ideal $I$ depends on the choice of generators of $I$. In our example, we can replace $x^3 - y^2$ with $x^4 - xy^2$, $x^3y - y^3$. We obtain the ideal $$\hat I = (xy, x^4 - xy^2, x^3y - y^3) = (xy, x^4, y^3).$$ This gives us the algebra $$\hat A_2 = \KK[x,y]/(xy, x^4, y^3)$$ and its generating subspace $\hat U_2 = \langle x, y, x^2, x^3 - y^2\rangle$.

Let us prove that the algebras $A_2$ and $\hat A_2$ are not isomorphic. Indeed, suppose that there exists an isomorphism $\varphi\colon A_2 \to \hat A_2$ defined by
\begin{align*}
& x \mapsto a_1x + a_2y + a_3x^2 + a_4x^3 + a_5y^2,\\
& y \mapsto b_1x + b_2y + b_3x^2 + b_4x^3 + b_5y^2.
\end{align*}

Since $x^3 = y^2$ in $A_2$, we must have
$$a_1^3x^3 = b_1^2x^2 + b_2^2y^2$$
in $\hat A_2$, so $a_1 = b_1 = b_2 = 0$. But in this case the element $\varphi(y)$ lies in $\hat \mm_2^2$, which gives us a contradiction.

We see that the $H$-pair obtained by our construction of shrinking the ideal of relations might depend on the order of generators of this ideal.

\end{example}

\section{Existence of non-degenerate hypersurfaces admitting an induced additive action.} 
In this section, we prove the following proposition.

\begin{proposition}
    For any positive integers $n$ and $d$ such that $2\leq d \leq n$ there exists a non-degenerate hypersurface of degree $d$ in $\PP^n$ that admits an induced additive action.
\end{proposition}

\begin{proof}
    It suffices to give an example of a Gorenstein algebra $A$ with the maximal ideal $\mm$ such that $\dim A = n + 1$ and $\Soc A = \mm^d$ for all $n$ and $d$ satisfying the statement of the proposition. Our construction depends on the parity of the number $n - d$.

    Suppose that $n - d$ is even. Then we have $d = n - 2k$ for some non-negative integer $k$. If $k$ is positive, let $A_{n,k}$ be the commutative associative unital algebra given by generators $S_1, S_2, \ldots, S_{2k+1}$ subject to the following relations:

 \begin{center}$S_iS_j = 0$ for all $i,j$, $1\leq i\leq j\leq 2k+1$,
    \end{center}
    
    \begin{center}
    except $(i,j) = (1,2)$, $(3,4), \ldots, (2k-1, 2k)$, $(2k+1, 2k+1)$, and
    \end{center}
    
    \begin{center}$S_1S_2 = S_3S_4 = \ldots = S_{2k-1}S_{2k} = S_{2k+1}^{n-2k}.$
    \end{center}
    
    The collection of elements $1, S_1, \ldots, S_{2k}, S_{2k+1}, S_{2k+1}^2, \ldots S_{2k+1}^{n-2k}$ is clearly a basis of $A_{n,k}$, so $\dim A_{n,k} = n + 1$. The maximal ideal of $A_{n,k}$ is $m_{n,k}^A = \langle S_1, \ldots, S_{2k}, S_{2k+1}, S_{2k+1}^2, \ldots S_{2k+1}^{n-2k}\rangle$.

    Let us compute the ideal $\Soc A_{n,k}$. Suppose that
    $$a = x_1S_1 + x_2S_2 + \ldots x_{2k}S_{2k} + y_1S_{2k+1} + y_2S_{2k+1}^2 + \ldots + y_{n-2k}S_{2k+1}^{n-2k} \in \Soc A.$$
    
    Since $S_ia = 0$ for all $i=1,2, \ldots, 2k, 2k+1$, multiplying $a$ by $S_1, S_2, \ldots, S_{2k}$, we obtain that $$x_2 = x_1 = \ldots = x_{2k}=x_{2k-1} = 0.$$ Next, multiplying $a$ by $S_{2k+1}^{n-2k-1}, S_{2k+1}^{n-2k-2}, \ldots, S_{2k+1}$ successively, we obtain that $$y_1 = y_2 = \ldots = y_{n-2k-1} = 0.$$ 
    
    This means that $\Soc A = \langle S_{2k+1}^{n-2k}\rangle$. It remains to note that $(\mm_{n,k}^A)^{n-2k} = \langle S_{2k+1}^{n-2k}\rangle$. This holds since $n-2k\geq 2$ and any element of the form $S_{i_1}S_{i_2}\ldots S_{i_{n-2k}}$ is either zero (this happens when there are at least two elements among $S_{i_1}, S_{i_2}, \ldots, S_{i_{n-2k}}$ whose product equals zero; in particular, this happens if $n-2k > 2$ and not all of the elements $S_{i_1}, S_{i_2}, \ldots, S_{i_{n-2k}}$ are equal to $S_{2k+1}$), or $S_{2k+1}^{n-2k}$ (this happens if all of the elements $S_{i_1}, S_{i_2}, \ldots, S_{i_{n-2k}}$ are equal to $S_{2k+1}$ or if $n-2k = 2$ and $(i_1, i_2) = (2l-1, 2l)$ for some $l=1,2,\ldots, k$).
    %Here we used the fact that $n-2k\geq 2$.

    We need to treat separately the case $k = 0$, i.e. $d = n$. We define $A_{n,0}$ as the algebra generated by one element $S$ subject to the relation $S^{n+1} = 0$. Then $A_{n,0}$ is isomorphic to $\KK[x]/(x^{n+1})$ and it is clear that $\dim A_{n,0} = n + 1$ and $\Soc A_{n,0} = \langle x^n\rangle$, so the algebra $A_{n,0}$ satisfies our conditions.

    Now, suppose that $n-d$ is odd, i.e. $d = n - 2k + 1$, where $k$ is a positive integer not exceeding $[\frac{n - 1}{2}]$. Let $B_{n,k}$ be the commutative associative unital algebra given by generators $S_1, S_2, \ldots, S_{2k}$ subject to the following relations:

    \begin{center}$S_iS_j = 0$ for all $i,j$, $0\leq i\leq j\leq 2k$,
    \end{center}
    
    \begin{center}
    except $(i,j) = (1,2)$, $(3,4), \ldots, (2k-3, 2k-2)$, $(2k-1, 2k-1)$, $(2k, 2k)$, and
    \end{center}
    
    \begin{center}$S_1S_2 = S_3S_4 = \ldots = S_{2k-3}S_{2k-2} = S_{2k-1}^2 = S_{2k}^{n - 2k + 1}.$
    \end{center}

    The collection of elements $1, S_1, \ldots, S_{2k-1}, S_{2k}, S_{2k}^2, \ldots S_{2k}^{n-2k + 1}$ is clearly a basis of $B_{n,k}$, so $\dim B_{n,k} = n + 1$. The maximal ideal of $B_{n,k}$ is $\mathfrak m_{n,k}^B = \langle S_1, \ldots, S_{2k-1}, S_{2k}, S_{2k}^2, \ldots S_{2k}^{n-2k+1}\rangle$.

    The proof of the fact that $\Soc B_{n,k} = \langle S_{2k}^{n-2k}\rangle$ and $\mathfrak (m_{n,k}^B)^{n-2k} = \langle S_{2k}^{n-2k}\rangle$ is identical to the one in the case of even $n - d$. This concludes the proof of Proposition 3.1.
        
\end{proof}

A natural question is whether the hypersurface of degree $d$ in $\PP^n$ admitting an induced additive action is unique. In general, the answer is negative as we show in the examples below. However, this is true for $d = 2$ and $d = n$. Indeed, a non-degenerate hypersurface of degree~$2$ is a smooth quadric, which is unique up to isomorphism. For $d = n$, such a hypersurface corresponds to an $H$-pair $(A, U)$, where $A$ is a local finite-dimensional Gorenstein algebra of dimension $n + 1$ with the maximal ideal $\mm$ such that $\mm^n \ne 0$ and $\mm^{n+1} = 0$. Consider the sequence of subspaces of $A$:
$$A\supseteq \mm \supseteq \mm^2 \supseteq \ldots \supseteq \mm^{n-1} \supseteq \mm^n \supseteq 0.$$

The dimension of the subspaces in this sequence strictly decreases. Since $\dim A = n + 1$, it must decrease exactly by $1$ at each step. In particular, $\dim \mm/\mm^2 = 1$. Now, Nakayama's Lemma implies that every element $x\in \mm\setminus \mm^2$ generates the algebra $A$. Therefore, the algebra~$A$ is isomorphic to $\KK[x]/I$ for some nilpotent ideal $I\subseteq \KK[x]$. But any nilpotent non-zero ideal in $\KK[x]$ is of the form $(x^k)$ for some positive integer $k$. Since $\dim A = n$, we conclude that $A \cong \KK[x]/(x^{n+1})$. It remaines to prove that the $H$-pair $(A, U)$ does not depend on the choice of a subspace $U\subseteq \mm$ complementary to $\Soc A = \langle x^n\rangle$. In the end of this section, we prove that the same property holds for the algebra $A_5 = \KK[x,y,z]/(x^2, y^2, xz, yz, xy - z^3)$. For the algebra $A = \KK[x]/(x^{n+1})$, the proof is similar. The proof is not difficult but technical and rather tedious, so we are not going to give it here.

    As an example, we compute explicitly the equations of projective hypersurfaces corresponding to all Gorenstein local algebras of dimension $6$. In \cite[Chapter 1.1]{AZ}, a complete list of local finite-dimensional algebras up to dimension $6$ is given. There are $25$ non-isomorphic local algebras of dimension $6$, and there are six Gorenstein local algebras among them. These are the following:

\begin{center}
\begin{tabular}{ |c|c|c| }
 \hline
 № & Algebra & $d$\\
 \hline\hline
 1 & $A_1 = \KK[x]/(x^6)$ & $5$\\
 \hline
 2 & $A_2 = \KK[x,y]/(xy, x^4 - y^2)$ & 4\\
 \hline
 3 & $A_3 = \KK[x,y]/(xy, x^3 - y^3)$ & 3 \\
 \hline
  4 & $A_4 = \KK[x,y]/(x^3, y^2)$ & 3\\ 
 \hline
  5 & $A_5 = \KK[x,y,z]/(x^2, y^2, xz, yz, xy - z^3)$ & 3\\
 \hline
  6 & $A_6 = \KK[x_1, x_2, x_3, x_4]/(x_i^2 - x_j^2, x_ix_j, i\ne j)$ & 2\\ 
 \hline
\end{tabular}
\end{center}

Note that the algebras $A_1$, $A_2$, $A_5$ and $A_6$ are the algebras that we constructed in the proof of Proposition 3.1.

Let us compute explicitly the equations of the non-degenerate hypersurfaces in~$\PP^5$ corresponding to these algebras (for some choice of generating subspaces $U_i$). As before, $\pi_i$ stands for the canonical projection $\mm_i\to\mm_i/U_i$.

\begin{enumerate}[wide, labelwidth=!, labelindent=0pt]
    \item We start with $A_1 = \KK[x]/(x^6)$. The maximal ideal of $A_1$ is $\mm_1 = \langle x, x^2, x^3, x^4, x^5\rangle$ and the socle is $\Soc A_1 = \langle x^5\rangle$. Choose the generating subspace $U_1 = \langle x, x^2, x^3, x^4\rangle$. The equation of the corresponding hypersurface $X_1$ is computed as follows:
    
\begin{multline*}
    \pi_1(\ln(1 + t_1x + t_2x^2 + t_3x^3 + t_4x^4 + t_5x^5)) = \pi_1(t_1x + t_2x^2 + \ldots + t_5x^5 - \\
    - \frac{(t_1x + t_2x^2 + \ldots + t_5x^5)^2}{2} + \frac{(t_1x + t_2x^2 + \ldots + t_5x^5)^3}{3} - \frac{(t_1x + t_2x^2 + \ldots + t_5x^5)^4}{4} +\\
    + \frac{(t_1x + t_2x^2 + \ldots + t_5x^5)^5}{5}) = t_5 - t_1t_4 - t_2t_3 + t_1^2t_3 + t_1t_2^2 - t_1^3t_2 + \frac{t_1^5}{5},
\end{multline*}

so the homogeneous equation of $X_1$ is
$$f_1(z) = z_0^4z_5 - z_0^3z_1z_4 - z_0^2z_2z_3 + z_0^2z_1^2z_3 + z_0^2z_1z_2^2 - z_0z_1^3z_2 + \frac{z_1^5}{5} = 0.$$

Let us find the singular locus of $X_1$. The open orbit of the $\mathbb G_a^4$-action on $X_1$ is $$\{[z_0: z_1: z_2: z_3: z_4: z_5]\colon z_0 \ne 0\},$$ so for any singular point $[z_0: z_1: z_2: z_3: z_4: z_5]$ of $X_1$ we have $z_0 = 0$. It is easy to see that the equalities $\frac{\partial f_1}{\partial z_i} = 0$, $i = 1,2,3,4,5$ hold if and only if $z_1 = 0$. We obtain that the singular locus of $X_1$ is the following subspace:
$$\{[z_0: z_1: z_2: z_3: z_4: z_5]\colon z_0 = z_1 = 0\}.$$

It is known that a projective hypersurface is normal if and only if its singular locus has codimension at least $2$ (see \cite[Section 5.1, Chapter 2]{IShaf}). Since the singular locus of $X_1$ has codimension~$1$ in $X_1$, the hypersurface $X_1$ is not normal.
\bigskip

\item $A_2 = \KK[x,y]/(xy, x^4 - y^2)$. The maximal ideal of $A_2$ is $\mm_2 = \langle x, y, x^2, x^3, x^4 \rangle$ and the socle is $\Soc A_2 = \langle x^4\rangle$. Choose the generating subspace $U_2 = \langle x, y, x^2, x^3\rangle$. Now, we have

\begin{multline*}
    \pi_2(\ln(1 + t_1x + t_2y + t_3x^2 + t_4x^3 + t_5x^4)) = \pi_2(t_1x + t_2y + \ldots + t_5x^4 - \\
    - \frac{(t_1x + t_2y + \ldots + t_5x^4)^2}{2} + \frac{(t_1x + t_2y + \ldots + t_5x^4)^3}{3} - \frac{(t_1x + t_2y + \ldots + t_5x^4)^4}{4})= \\
    = t_5 - t_1t_4 - \frac{t_2^2}{2} - \frac{t_3^2}{2} + t_1^2t_3 - \frac{t_1^4}{4},
\end{multline*}
so the homogeneous equation of the corresponding hypersurface $X_2$ is

$$z_0^3z_5 - z_0^2z_1z_4 - \frac{z_0^2z_2^2}{2} - \frac{z_0^2z_3^2}{2} + z_0z_1^2z_3 - \frac{z_1^4}{4} = 0.$$

It is easy to see that the singular locus of $X_2$ is
$$\{[z_0: z_1: z_2: z_3: z_4: z_5]\colon z_0 = z_1 = 0\}.$$
Since it has codimension $1$ in $X_2$, the hypersurface $X_2$ is not normal.

\bigskip

\item $A_3 = \KK[x,y]/(xy, x^3 - y^3)$. The maximal ideal of $A_3$ is $\mm_3 = \langle x, y, x^2, y^2, x^3\rangle$ and the socle is $\Soc A_3 = \langle x^3\rangle$. Choose the generating subspace $U_3 = \langle x, y, x^2, y^2\rangle$. We have

\begin{multline*}
    \pi_3(\ln(1 + t_1x + t_2y + t_3x^2 + t_4y^2 + t_5x^3)) = \pi_3(t_1x + t_2y + t_3x^2 + t_4y^2 + t_5x^3 -\\
    - \frac{(t_1x + t_2y + t_3x^2 + t_4y^2 + t_5x^3)^2}{2} + \frac{(t_1x + t_2y + t_3x^2 + t_4y^2 + t_5x^3)^3}{3}) = \\
    = t_5 - t_1t_3 - t_2t_4 + \frac{t_1^3}{3} + \frac{t_2^3}{3},
\end{multline*}
so the homogeneous equation of the corresponding projective hypersurface $X_3$ is

$$z_0^2z_5 - z_0z_1z_3 - z_0z_2z_4 + \frac{z_1^3}{3} + \frac{z_2^3}{3} = 0.$$

The singular locus of $X_3$ is
$$\{[z_0: z_1: z_2: z_3: z_4: z_5]\colon z_0 = z_1 = z_2 = 0\}.$$

It has codimension $2$ in $X_3$, so the hypersurface $X_3$ is normal.

\bigskip

\item $A_4 = \KK[x,y]/(x^3, y^2)$. The maximal ideal of $A_4$ is $\mm_4 = \langle x, y, x^2, xy, x^2y \rangle $ and the socle is $\Soc A_4 = \langle x^2y\rangle$. Choose the generating subspace $U_4 = \langle x, y, x^2, xy\rangle$. Then

\begin{multline*}
    \pi_4(\ln(1 + t_1x + t_2y + t_3x^2 + t_4xy + t_5x^2y)) = \pi_4(t_1x + t_2y + t_3x^2 + t_4xy + t_5x^2y -\\
    - \frac{(t_1x + t_2y + t_3x^2 + t_4xy + t_5x^2y)^2}{2} + \frac{(t_1x + t_2y + t_3x^2 + t_4xy + t_5x^2y)^3}{3}) = \\
    = t_5 - t_1t_4 - t_2t_3 + t_1^2t_2, 
\end{multline*}
    so the homogeneous equation of the corresponding projective hypersurface $X_4$ is
    
$$z_0^2z_5 - z_0z_1z_4 - z_0z_2z_3 + z_1^2z_2 = 0.$$

The singular locus of $X_4$ is
$$\{[z_0: z_1: z_2: z_3: z_4: z_5]\colon z_0 = z_1 = 0\}.$$

It has codimension~$1$ in $X_4$, so the hypersurface $X_4$ is not normal.

\bigskip

\item $A_5 = \KK[x,y,z]/(x^2, y^2, xz, yz, xy - z^3)$. The maximal ideal of $A_5$ is $\mm_5 = \langle x, y, z, z^2, z^3\rangle$ and the socle is $\Soc A_5 = \langle z^3\rangle$. Choose the generating subspace $U_5 = \langle x, y, z, z^2 \rangle$. Then we have

\begin{multline*}
    \pi_5(\ln(1 + t_1x + t_2y + t_3z + t_4z^2 + t_5z^3)) = \pi_5(t_1x + t_2y + t_3z + t_4z^2 + t_5z^3 -\\
    - \frac{(t_1x + t_2y + t_3z + t_4z^2 + t_5z^3)^2}{2} + \frac{(t_1x + t_2y + t_3z + t_4z^2 + t_5z^3)^3}{3} = t_5 - t_1t_2 - t_3t_4 + \frac{t_3^3}{3},
\end{multline*}
so the homogeneous equation of the corresponding projective hypersurface $X_5$ is

$$z_0^2z_5 - z_0z_1z_2 - z_0z_3z_4 + \frac{z_3^3}{3} = 0.$$

The singular locus of $X_5$ is
$$\{[z_0: z_1: z_2: z_3: z_4: z_5]\colon z_0 = z_3 = 0\}.$$
It has codimension $1$ in $X_5$, so the hypersurface $X_5$ is not normal.

\bigskip

\item $A_6 = \KK[x_1, x_2, x_3, x_4]/(x_i^2 - x_j^2, x_ix_j, i\ne j)$.  The maximal ideal of $A_6$ is $\mm_6 = \langle x_1, x_2, x_3, x_4, x_1^2 \rangle$ and the socle is $\Soc A_6 = \langle x_1^2 \rangle$. Choose the generating subspace $U_6 = \langle x_1, x_2, x_3, x_4 \rangle$. Then

\begin{multline*}
\pi_6(\ln(1 + t_1x_1 + t_2x_2 + t_3x_3 + t_4x_4 + t_5x_1^2)) = \pi_6(t_1x_1 + t_2x_2 + t_3x_3 + t_4x_4 + t_5x_1^2 - \\
- \frac{(t_1x_1 + t_2x_2 + t_3x_3 + t_4x_4 + t_5x_1^2)^2}{2}) = t_5 - \frac{t_1^2}{2} - \frac{t_2^2}{2} - \frac{t_3^2}{2} - \frac{t_4^2}{2},
\end{multline*}
so the homogeneous equation of the corresponding projective hypersurface $X_6$ is

$$z_0z_5 - \frac{z_1^2}{2} - \frac{z_2^2}{2} - \frac{z_3^2}{2} - \frac{z_4^2}{2} = 0.$$

The hypersurface $X_6$ is a smooth quadric, so it is automatically normal.

\end{enumerate}

\bigskip

To conclude, we obtained the following six non-degenerate pairwise non-isomorphic projective hypersurfaces in $\PP^5$ admitting an induced additive action:

\bigskip

\renewcommand{\arraystretch}{1.5}
\hspace{-2.2cm}
\begin{tabular}{ |c|c|c|c| }
 \hline
 Equation & Degree & Singular locus & Is it normal?\\
 \hline\hline
 $z_0^4z_5 - z_0^3z_1z_4 - z_0^2z_2z_3 + z_0^2z_1^2z_3 + z_0^2z_1z_2^2 - z_0z_1^3z_2 + \frac{1}{5}z_1^5 = 0$ & 5 & $\{z_0 = z_1 = 0\}$ & No\\
 \hline
 $z_0^3z_5 - z_0^2z_1z_4 - \frac{1}{2}z_0^2z_2^2 - \frac{1}{2}z_0^2z_3^2 + z_0z_1^2z_3 - \frac{1}{4}z_1^4=0$ & 4 & $\{z_0 = z_1 = 0\}$ & No\\ 
 \hline
  $z_0^2z_5 - z_0z_1z_3 - z_0z_2z_4 + \frac{1}{3}z_1^3 + \frac{1}{3}z_2^3 = 0$ & 3 & $\{z_0 = z_1 = z_2 = 0\}$ & Yes \\ 
 \hline
  $z_0^2z_5 - z_0z_1z_4 - z_0z_2z_3 + z_1^2z_2 = 0$ & 3 & $\{z_0 = z_1 = 0\}$ & No\\ 
 \hline
  $z_0^2z_5 - z_0z_1z_2 - z_0z_3z_4 + \frac{1}{3}z_3^3 = 0$ & 3 & $\{z_0 = z_3 = 0\}$ & No\\
 \hline
  $z_0z_5 - \frac{1}{2}z_1^2 - \frac{1}{2}z_2^2 - \frac{1}{2}z_3^2 - \frac{1}{2}z_4^2 = 0$ & 2 & $\O$ & Yes\\ 
 \hline
\end{tabular}

\renewcommand{\arraystretch}{1}

\bigskip

We see that for $d = 2, 4, 5$ there is a unique, up to isomorphism, non-degenerate hypersurface of degree $d$ in $\PP^5$ admitting an induced additive action, while for $d = 3$ there are three isomorphism classes of such hypersurfaces.

Strictly speaking, we cannot say yet that these are all non-degenerate hypersurfaces in~$\PP^5$ admitting an induced additive action. The problem is that $H$-pairs of the form $(A, U)$ might be non-equivalent for different choices of a subspace $U$ complementary to $\Soc A$. However, it can be shown by explicit computations that this does not happen for the algebras $A_1$, $A_2$, $A_3$, $A_4$, $A_5$, $A_6$. Let us do this, for instance, for the algebra $A_5$. It suffices to show that the $H$-pair $(A_5, U_5)$ is equivalent to any $H$-pair $(A_5, U_5')$ for any hyperplane $U_5'\subseteq \mm_5$ complementary to~$\langle z^3 \rangle$. Any such subspace $U_5'$  consists of the elements of the form
$$t_1x + t_2y + t_3z + t_4z^2 + t_5z^3$$
satisfying the equation

$$p_1t_1 + p_2t_2 + p_3t_3 + p_4t_4 + t_5 = 0$$
for some $p_1, p_2, p_3, p_4\in \KK$. We need to show that there exists an algebra automorphism $\varphi\colon A_5 \to A_5$ sending $U_5$ to~$U_5'$. Let us try to find such an isomorphism in the form
\begin{align*}
    & x \mapsto x + a_1z^2 + a_2z^3,\\
    & y \mapsto y + b_1z^2 + b_2z^3,\\
    & z \mapsto z + c_1z^2 + c_2z^3
\end{align*}
for some $a_1, a_2, b_1, b_2, c_1, c_2 \in \KK$.

It is clear that the ideal of relations $I_5 = (x^2, y^2, xz, yz, xy - z^3)$ is mapped by any such $\varphi$ to itself, so $\varphi$ is a well-defined homomorphism from $A_5$ to itself. We only need to ensure that the subspace $U_5$ is mapped to $U_5'$ by $\varphi$. Recall that $U_5 = \langle x, y, z, z^2\rangle$. Since $\varphi(z^2) = z^2 + 2c_1z^3$, we must have
$$p_4 + 2c_1 = 0, \hspace{0.2cm} c_1 = -\frac{p_4}{2}.$$

Now, the condition $\varphi(x), \varphi(y), \varphi(z) \in U_5'$ is satisfied if
$$p_1 + a_1p_4 + a_2 = p_2 + b_1p_4 + b_2 = p_3 + c_1p_4 + c_2 = 0.$$

We can choose $a_1, b_1$ arbitrarily and then $a_2, b_2, c_2$ are defined uniquely.

To finish the proof, we only need to show that $\varphi$ is bijective. But the elements $$x + a_1z^2 + a_2z^3, y + b_1z^2 + b_2z^3, z + c_1z^2 + c_2z^3$$ form a basis of $\mm_5/\mm_5^2$ as a vector space over $\KK$. Thus, these three elements generate $A_5$ by Nakayama's Lemma. It follows that $\varphi$ is surjective, so it is bijective for dimension reasons.

Using the same technique, one can prove that the hypersurfaces corresponding to $(A_i, U)$ do not depend on the choice of $U$ for any $i = 1, 2, \ldots, 6$ (again, provided that $U\subseteq \mm_i$ is complementary to $\Soc A_i$). Therefore, there are exactly six non-degenerate hypersurfaces in $\PP^5$ admitting an additive action up to equivalence.


\begin{thebibliography}{99}

\bibitem{AP} Ivan Arzhantsev and Andrey Popovskiy. Additive actions on projective hypersurfaces. In: Automorphisms in Birational and Affine Geometry, Springer Proc. Math. Stat., 79, Springer, 2014, 17-33

\bibitem{ASh} Ivan Arzhantsev and Elena Sharoyko. Hassett–Tschinkel correspondence: Modality and projective hypersurfaces. J. Algebra 348 (2011), no. 1, 217-232

\bibitem{AZ} Ivan Arzhantsev and Yulia Zaitseva. Equivariant completions of affine spaces. Russian Math. Surveys. 77 (2022), no. 4, 571–650

\bibitem{BGT} Viktoriia Borovik, Sergey Gaifullin, and Anton Trushin. Commutative actions on smooth projective quadrics. Comm. Algebra. 50 (2022), no. 12, 5468-5476

\bibitem{HT} Brendan Hassett and Yuri Tschinkel. Geometry of equivariant compactifications of $\mathbb G_a^n$. Int. Math. Res. Not. IMRN 1999 (1999), no. 22, 1211-1230

\bibitem{YL} Yingqi Liu. Additive actions on hyperquadrics of corank two. Electron. Res. Arch. 30 (2022), no. 1, 1-34

\bibitem{Shaf} Anton Shafarevich. Additive actions on toric projective hypersurfaces. Results Math. 76 (2021), no. 3, art. 145

\bibitem{IShaf} Igor Shafarevich. Basic algebraic geometry /I. R. Shafarevich; translated from the Russian by K. A. Hirsch. Springer-Verlag Berlin; New York, 1974

\bibitem{Sh} Elena Sharoyko. Hassett-Tschinkel correspondence and automorphisms of the quadric. Sb. Math. 200 (2009), no. 11, 1715–1729

%\bibitem{AP} I. Arzhantsev, A. Popovskiy (2014). Additive Actions on Projective Hypersurfaces. In: Cheltsov, I., Ciliberto, C., Flenner, H., McKernan, J., Prokhorov, Y., Zaidenberg, M. (eds) Automorphisms in Birational and Affine Geometry. Springer Proceedings in Mathematics I\& Statistics, vol 79. Springer, Cham. https://doi.org/10.1007/978-3-319-05681-4\_2

\end{thebibliography}
\end{document}